\documentclass[11pt,twoside]{amsart}

      \usepackage{amssymb}
      \usepackage{graphicx}
      \usepackage{dcolumn,indentfirst,cite,color,hyperref}

      \theoremstyle{plain}
     \newtheorem{thm}{Theorem}[section]
\newtheorem{lem}[thm]{Lemma}

\newtheorem{cor}[thm]{Corollary}

\newcommand{\bess}{\begin{eqnarray*}}
\newcommand{\eess}{\end{eqnarray*}}
\input xy
\xyoption{all}

\usepackage{fancyhdr} %页眉和页脚
\begin{document}

\author{Xiaoguang Wang}
\address{Department of Mathematics, Zhejiang University, Hangzhou, 310027, P.R.China}
\email{wxg688@163.com}

% \author{Yongcheng Yin}
% \address{Department of Mathematics, Zhejiang University, Hangzhou, 310027, P.R.China}
% \email{yin@zju.edu.cn}

%%%%%%%%%%%%%%%%%%%%%%%%%%%%%%%%%%%%%%%%%%%%%%%%%%%%%

   % the address where the research was carried out
   %\address{University of G\"ottingen, G\"ottingen, Germany}

   % current address, usually not needed because it is the same as the
   % regular address
   %\curraddr{Department of Mathematics, Pennsylvania State University,
    % University Park, State College PA 16802}

   %\email{simpson@math.psu.edu}

   % title

   \title[]{Conjugacy between polynomial basins}

   % Note that the short title for running heads goes in square
   % brackets.  This is optional.  The long title goes in curly
   % braces.  In the long title, line breaks are indicated by \\.

   % abstract (optional)
   \begin{abstract}
 In this article, we study the  properties of conjugacies between polynomial basins.
 For any conjugacy, there is a quasiconformal conjugacy in the same homotopy class minimizing the
 dilatation. We compute the precise value of the minimal dilatation.
  The quasiconformal conjugacy  minimizing  the dilatation  is not unique in general. We give a necessary and sufficient condition when the
 extremal map is unique.
   \end{abstract}

   % AMS subject classifications (used in AMS journals)
   \subjclass[2010]{Primary 37F10; Secondary 37F30}

   % AMS keywords (used in AMS journals)
   \keywords{basin of infinity, conjugacy}

   % acknowledge support, etc
  % \thanks{This research was partially supported by NSF grant
   %  DOA-123456789.}
   %\thanks{We would like to thank our colleagues for their helpful
    % criticism.}

   % dedication
  % \dedicatory{Dedicated to Professor Donald Knuth on the occasion
   %  of his $100$th birthday}

   % today's date, or fill in whatever date you prefer
   \date{\today}

% This ends the top matter information.
% We can now tell LaTeX to display the top matter.

   \maketitle

\section{Introduction}

Let $f$ be a polynomial of degree $d\geq2$. The complex plane can be decomposed into the union of its open and connected {\it basin of
infinity} % defined by
$$X(f):=\{z\in\mathbb{C}: f^n(z)\rightarrow\infty \text{ as } n\rightarrow\infty \}$$
and its complement, the compact {\it filled Julia set} $K(f)$.
%The filled Julia set $K(f)$ is defined by
%$$K(f)=\{z\in \mathbb{C}: \text{ the sequence } f^n(z) \text{ is bounded }\}$$
The {\it Green function} $G_f: \mathbb{C}\rightarrow [0,\infty)$ of $f$ is given by
$$G_f(z)=\lim_{n\rightarrow\infty}d^{-n}\log^+|f^n(z)|,$$
where $\log^+(t)=\max\{0,\log t\}$ for $t\in[0,\infty)$. The function $G_f$ vanishes exactly on $K(f)$, harmonic on $X(f)$, continuous on
$\mathbb{C}$, and satisfies $G_f(f(z))=dG_f(z)$ for all $z\in \mathbb{C}$.

 We denote by $C(f)$  the critical set of $f$ in $\mathbb{C}$, and $P(f):=\overline{\cup_{k\geq1}f^k(C(f))}$ the postcritical set.
Consider the dynamic of $f$ restricted in  $X(f)$.
We say that $\phi$ is a conjugacy between  $(f,X(f))$ and $(g, X(g))$ if $\phi:X(f)\rightarrow X(g)$ is a homeomorphism  such that $\phi\circ f=
g\circ \phi$ in $X(f)$.
It determines a homotopy class  $[\phi]$  consisting of all  conjugacies between   $(f,X(f))$ and $(g, X(g))$, homotopic to  $\phi$  rel the set $P(f)\cap X(f)$. %
It's known that the homotopy class $[\phi]$ always contains  a quasiconformal (or smooth) map.

Let $\phi$ be a quasiconformal conjugacy between  $(f,X(f))$ and $(g, X(g))$, let $$\mu_\phi=\frac{\overline{\partial}\phi}{\partial \phi}=\frac{\phi_{\overline{z}}}{\phi_z}\frac{d\overline{z}}{dz}$$ be its Beltrami form.
The esssup norm $\|\mu_\phi\|$ of $\mu_\phi$ measures  the  dilatation of $\phi$.
%There are two quantities measuring  the  dilatation of $\phi$:  the esssup norm $\|\mu_\phi\|$ of $\mu_\phi$, and $K(\phi)=\frac{1+\|\mu_\phi\|}{1-\|\mu_\phi\|}$.
%In this paper, we use the first as dilatation.
  Let $\delta_{[\phi]}=\inf\{\|\mu_\psi\|: \psi\in[\phi] \text{ is a quasiconformal map} \}$.  Note that the  quasiconformal maps in $[\phi]$ form a normal family,  there is a  quasiconformal map $\varphi\in[\phi]$ whose dilatation equals $\delta_{[\phi]}$. The natural questions are:

  \

 {\it  Suppose $\phi$ is a conjugacy between  $(f,X(f))$ and $(g, X(g))$, what is  $\delta_{[\phi]}$? Further, let $\varphi\in [\phi]$ with $\|\mu_\varphi\|=\delta_{[\phi]}$, when is $\varphi$ unique?  }

\

 This note is devoted to answer these questions.
 We compute the precise value  $\delta_{[\phi]}$ and give a necessary and sufficient condition when the extremal map is unique. Since some basic properties
  of conjugacies are needed first, we state the result in the next section.
  The main result is  Theorem \ref{main}. 

 This work is inspired by DeMarco and Pilgrim's sequel works \cite{DP1,DP2,DP3}. The notations we use here are borrowed from their's.  The original idea concerning the deformation of polynomial basins comes from the {\it wring motion }(also called the {\it Branner-Hubbard motion}) introduced by Branner and Hubbard \cite{BH1}, and the Teichm\"uller theory of rational maps developed by McMullen and Sullivan \cite{MS}. The uniqueness of the extremal quasiconformal conjugacy follows from a result of Reich and
 Strebel \cite{RS}. The theory of  extremal quasiconformal  mappings is used by Cui \cite{C} to construct  the extremal conjugacy between rational maps.
 
%  The extremal  conjugacy between  rational maps via  is studied by Cui \cite{C}. 

 We denote by  $\mathbb{D}$  the unit disk,  $\mathbb{H}=\{\Re z>0\}$ the right half plane endowed the metric $|dz|/\Re z$.

%If $(f,X(f))$ and $(g, X(g))$  are topologically conjugate, then there is a
%q.c. conjugacy between them. We may label the critical points by $c_f^i$ and $c_g^i$, $1\leq i\leq D$, so that  a conjugacy maps $c_f^i$ to
%$c_g^i$.

% If $(f,X(f))$ and $(g, X(g))$  are topologically conjugate, then there is a
%q.c. conjugacy between them. We define
%$$\delta((f,X(f)), (g,X(g)))=\inf\{\|\mu_h\|: h \text { is  a qc conjugacy between } (f,X(f)) \text{ and }(g, X(g))\}.$$

%It's natural to know the quality of $\delta((f,X(f), (g,X(g))$.

%\begin{thm} $$\delta((f,X(f), (g,X(g))=\max_{1\leq i\leq D}\bigg|\frac{{\rm Log} \phi_f(f^{l(c_f^i)}(c_f^i))-
%{\rm Log} \phi_g(g^{l(c_g^i)}(c_g^i))}{{\rm Log}\phi_f(f^{l(c_f^i)}(c_f^i))+\overline{{\rm Log} \phi_g(g^{l(c_g^i)}(c_g^i))}}\bigg|.$$
%\end{thm}

\section{Basic properties of conjugacy}

%Let $\mathcal{P}_d$ be the space of monic and centered polynomials of degree $d$. Let $f\in \mathcal{P}_d$ and
%

Let $\mathcal{P}_d$ be the space of monic and centered polynomials. That is each $f\in \mathcal{P}_d$ takes the form
$$f(z)=z^d+a_{d-2}z^{d-2}+\cdots+a_1z+a_0,$$
where $a_{d-2},\cdots,a_0\in\mathbb{C}$.  Any polynomial of degree $d$ is  conjugate by an automorphism of $\mathbb{C}$ to some element of
$\mathcal{P}_d$.

Set $M(f)=\max\{G_f(c): c\in C(f)\}.$
The B\"ottcher map $\phi_f$ of $f$ is defined near $\infty$  with $\phi'_f(\infty)=1$  and it can be extended to an isomorphism  (c.f. \cite{M})
$$\phi_f: \{z\in\mathbb{C}: G_f(z)>M(f)\}\rightarrow \{w\in \mathbb{C}:|w|>e^{M(f)}\}.$$

%Let $C(f)$ be the critical set  of $f$ in $\mathbb{C}$.

\subsection{Conjugacy preserves escape order}

%First, a distortion lemma for escape rate:

 \begin{lem}\label{level} Let $\phi$ be a conjugacy  between   $(f,X(f))$ and $(g, X(g))$, then for any $z_1,z_2\in X(f)$,

 1.
 $G_f(z_1)>G_f(z_2)$ if and only if $G_g(\phi(z_1))>G_g(\phi(z_2))$.

 2. If $\phi$ is quasiconformal, then
 $$\frac{G_f(z_1)-G_f(z_2)}{K(\phi)}\leq G_g(\phi(z_1))-G_g(\phi(z_2))\leq{K(\phi)}(G_f(z_1)-G_f(z_2)).$$
\end{lem}
\begin{proof} We assume $G_f(z_1)>G_f(z_2)$. Choose  $k\geq0$  such that $G_f(f^k(z_2))\geq M(f)$,
then the set $A=\{z\in \mathbb{C}: G_f(f^k(z_2))<G_f(z)<G_f(f^k(z_1))\}$
is an annulus.
%We denote the inner boundary and outer boundary of $\varphi(A)$  by $\gamma_{1}$ and  $\gamma_{2}$.
Since
${\rm mod}(\phi_gg^n\phi(A))={\rm mod}(g^n\phi(A))=d^n{\rm mod}(\phi(A))\rightarrow\infty$
as $n\rightarrow\infty$, we conclude by (\cite{Mc}, Theorem 2.1) that
when $n$ is large enough,
$$ \max_{p\in \alpha_n}|p|<\min_{q\in \beta_n}|q|,$$
 where $\alpha_n$ and $\beta_n$ are inner and outer boundaries of $\phi_gg^n\phi(A)$, respectively.
 Since $\phi_gg^{n+k}(\phi (z_1))\in \beta_n$ and $\phi_gg^{n+k}(\phi(z_2))\in \alpha_n$, we have that
  $|\phi_gg^{n+k}(\phi (z_1))|>|\phi_gg^{n+k}(\phi (z_2))|$. Thus $G_g(\phi (z_1))>G_g(\phi (z_2))$. The `if' part follows from the same argument,
  by considering $\phi^{-1}$.

To prove the bi-Lipschitz inequality when $\phi$ is quasiconformal, we first observe that
$G_f(z_1)=G_f(z_2)$ if and only if $G_g(\phi(z_1))=G_g(\phi(z_2))$ by  the above conclusion. We may still assume $G_f(z_1)>G_f(z_2)$,
then %$\phi$ maps the equipotential set $\{G_f(z)=L_1\}$ to some equipotential set $\{G_g(z)=L_2\}$.
 $$\phi(A)=\{w\in \mathbb{C}: G_g(g^k(\phi(z_2)))<G_g(w)<G_g(g^k(\phi(z_1)))\}.$$

 Note that $2\pi {\rm mod}(\phi(A))=d^{k}$$(G_g(\phi(z_1))-G_g(\phi(z_2)))$ and
$2\pi {\rm mod}(A)=d^{k}(G_f(z_1)-G_f(z_2))$.
By the modular distortion
$${K(\phi)}^{-1}{\rm mod}(A)\leq {\rm mod}(\phi(A))\leq {K(\phi)}{\rm mod}(A),$$
 we get the required inequality.
\end{proof}

Given a polynomial  $f\in\mathcal{P}_d$, the {\it DeMarco-McMullen tree} $T(f)$ of $f$ is the quotient of $X(f)$ obtained by collapsing each connected
component of a level set of $G_f$ to a single point, see \cite{DM}.
Let $\pi_f: X(f)\rightarrow T(f)$ be the quotient map, then $f$ induces a tree map $\sigma_f:T(f)\rightarrow T(f)$ by $\sigma_f\circ \pi_f=
\pi_f\circ f$. The function $G_f$ descends to a height function $h_f: T(f)\rightarrow \mathbb{R}_+$ by $G_f=h_f\circ\pi_f$, which induces a metric $d_f$ on $T(f)$. A consequence of Lemma \ref{level} is 

\begin{cor}\label{tree} A conjugacy $\phi$ between   $(f,X(f))$ and $(g, X(g))$ induces an isomorphism $\iota_\phi: (T(f), \sigma_f, d_f)\rightarrow (T(g), \sigma_g, d_g)$
preserving the tree dynamics. If $\phi$ is quasiconformal, then $\iota_\phi$ is $K(\phi)$-isometric:
 $$\frac{d_f(x,y)}{K(\phi)}\leq d_g(\iota_\phi(x),\iota_\phi(y))\leq K(\phi)d_f(x,y), \   \forall x,y \in T(f).$$
\end{cor}

\subsection{Conjugacy preserves angular difference}

The 1-form $\omega_f=2\partial G_f=2\frac{\partial G_f}{\partial z}dz$ satisfies $f^*\omega_f=d\omega_f$ and it
  provides a dynamically determined conformal metric $|\omega_f|$ on $X(f)$, with singularities at the escaping critical points and their inverse preimages (c.f. \cite{DM,DP2,DP3}). In the metric $|\omega_f|$, for any $L>0$,  we have
  $$\int_{G_f=L}|\omega_f|=2\pi.$$

  \begin{lem}\label{angular} Let $\phi$ be a conjugacy between   $(f,X(f))$ and $(g, X(g))$. Then for any $L>0$ and any connected arc $\gamma$ on  the level set $\{G_f=L\}$, we have
  $$\int_{\gamma}|\omega_f|=\int_{\phi(\gamma)}|\omega_g|.$$
\end{lem}
\begin{proof} For  $k\geq0$, let $I_k(f,\gamma)$ (resp.  $I_k(g,\phi(\gamma))$) be the integer part of $\frac{1}{2\pi}\int_{\gamma}|(f^k)^*\omega_f|$ (resp. $\frac{1}{2\pi}\int_{\phi(\gamma)}|(g^k)^*\omega_g|$). Since $\phi$ is a conjugacy,  we have $I_k(f,\gamma)=I_k(g,\phi(\gamma))$ for all $k\geq0$. Thus
$$\int_{\gamma}|\omega_f|=\lim_{k\rightarrow\infty} \frac{2\pi I_k(f,\gamma)}{d^k}=\lim_{k\rightarrow\infty} \frac{2\pi I_k(g,\phi(\gamma))}{d^k}=\int_{\phi(\gamma)}|\omega_g|.$$ \hfill \end{proof}

\subsection{Branner-Hubbard  motion}

In the rest of the paper, we assume $C(f)\cap X(f)\neq \emptyset$. In that case
$M(f)>0$. The fundamental annulus is
$A(f)=\{z\in\mathbb{C}: M(f)<G_f(z)\leq dM(f)\}$.  Set $\ell_0=M(f)$. The critical orbits meet $A(f)$ at the level sets
$\{G_f=\ell_j\}, 1\leq j \leq N$ with $\ell_0<\ell_1<\cdots<\ell_N=d\ell_0$. These level sets decompose  $A(f)$  into $N$ subannuli
$A_j=\{z\in\mathbb{C}:\ell_{i-1}<G_f(z)<\ell_j\}, 1\leq j \leq N$. Note that some level set $\{G_f=\ell_k\}$ may meet at least two critical orbits.

By Lemma \ref{level}, any conjugacy $\phi$ between $(f,X(f))$ and $(g, X(g))$ can induce a map $S_\phi: \mathbb{R}_+\rightarrow\mathbb{R}_+$ such that
$S_\phi\circ G_f=G_g\circ \phi$ on $X(f)$. It is  monotone increasing and satisfies $S_\phi(dt)=dS_\phi(t)$ for all $t>0$.

We remark that by identifying $t$ and $dt$ in $\mathbb{R}_+$, the map $S_\phi$ descends to  a circle homeomorphism  $\tau_\phi: \mathbb{S}\rightarrow\mathbb{S}$.

By Lemma \ref{angular}, when $t>M(f)$, the angular difference between $\phi_f(z)$ and $\phi_g(\phi(z))$ is a constant for all $z\in \{G_f=t\}$.
 Thus  $\phi$ induces a map $T_\phi: (M(f),+\infty)\rightarrow\mathbb{R}$ such that  $T_\phi(G_f(z))=\arg\phi_g(\phi(z))-\arg\phi_f(z)$ when $G_f(z)>M(f)$. The map $T_\phi$ satisfies $T_\phi(dt)=dT_\phi(t)$ for all $t>M(f)$. We can extend this map to a map from $\mathbb{R}_+$ to $\mathbb{R}$, still denoted by $T_\phi$, such that  $T_\phi(dt)=dT_\phi(t)$ for all $t>0$. Following Branner and Hubbard \cite{BH1}, we call $S_\phi$ the {\it stretching} part of $\phi$ and $T_\phi$ the {\it turning} part of $\phi$.

 Note  that for any two conjugacies $\phi,\varphi$ between   $(f,X(f))$ and $(g, X(g))$, one has   $S_\phi(\ell_j)=S_\varphi(\ell_j), \ \forall 1\leq j \leq N.$
 So the homotopy classes of conjugacies are determined by their turning parts.
The following fact is immediate:

\begin{lem}\label{c-equi} Let $\phi,\varphi$ be two conjugacies between   $(f,X(f))$ and $(g, X(g))$. Then  $[\phi]=[\varphi]$ if and only if
  $$T_\phi(\ell_j)=T_\varphi(\ell_j), \ \forall 1\leq j \leq N .$$
\end{lem}

In fact $T_\phi(\ell_j)-T_\varphi(\ell_j)\in \frac{2\pi}{k_j}\mathbb{Z}$ for some integer $k_j\geq1$, which is the fold of  symmetry of the level set $\{G_f=\ell_j\}$.

In the following, we assume $\phi$ is  a quasiconformal conjugacy between   $(f,X(f))$ and $(g, X(g))$, then $\phi$ is differentiable almost everywhere in $X(f)$. This implies
$S'_\phi(t)$ and $T'_\phi(t)$ exist for a.e $t\in\mathbb{R}_+$. By Lemma \ref{level}, we have  $K(\phi)^{-1}\leq S'_\phi(t)\leq K(\phi)$ a.e $t\in \mathbb{R}_+$.

\begin{lem}\label{beltrami} Let $\phi$ be a quasiconformal conjugacy between   $(f,X(f))$ and $(g, X(g))$. Then we have
%$$\frac{\overline{\partial }\phi}{\partial \phi}=\frac{Z(G_f(z))-1}{Z(G_f(z))+1}\frac{\overline{\partial }G_f}{\partial G_f},$$
%where $Z:\mathbb{R}_+\rightarrow \mathbb{H}$ is defined by $Z(t)=S'_\phi(t)+iT'_\phi(t)$.
 $$\frac{\overline{\partial }\phi}{\partial \phi}=\frac{S_\phi'\circ G_f+iT_\phi'\circ G_f-1}{S_\phi'\circ G_f+iT_\phi'\circ G_f+1}\frac{\overline{\partial }G_f}{\partial G_f}.$$
\end{lem}

\begin{proof} It suffices to verify the equation in $\{z\in \mathbb{C}: G_f(z)>M(f)\}$ since both sides are $f$-invariant.

When $G_f(z)>M(f)$, we have
%$$\phi_g\circ\phi \circ \phi_f^{-1}(e^{G_f(z)+i\theta})=e^{S_\phi(G_f(z))+i\theta+iT_\phi(G_f(z))}.$$
%Note that  $z=\phi_f^{-1}(e^{G_f(z)+i\theta})$, then $e^{i\theta}=\phi_f(z)/|\phi_f(z)|$ and
$$\phi_g(\phi(z))=e^{S_\phi(G_f(z))+iT_\phi(G_f(z))}\frac{\phi_f(z)}{|\phi_f(z)|}.$$

Set $E(z)=e^{S_\phi(G_f(z))+iT_\phi(G_f(z))}$, then
$$\frac{\partial (\phi_g\circ\phi)}{\partial \overline{z}}=E(z)\Big[(S'_\phi(G_f(z))+iT'_\phi(G_f(z)))
\frac{\phi_f(z)}{|\phi_f(z)|}\frac{\partial G_f}{\partial \overline{z}}+\frac{\partial}{\partial \overline{z}}\Big(\frac{\phi_f(z)}{|\phi_f(z)|}\Big)\Big],$$
$$\frac{\partial (\phi_g\circ\phi)}{\partial {z}}=E(z)\Big[(S'_\phi(G_f(z))+iT'_\phi(G_f(z)))
\frac{\phi_f(z)}{|\phi_f(z)|}\frac{\partial G_f}{\partial {z}}+\frac{\partial}{\partial z}\Big(\frac{\phi_f(z)}{|\phi_f(z)|}\Big)\Big].$$

By calculation,
%$$\frac{\partial G_f(z)}{\partial \overline{z}}=\frac{\overline{\phi'_f(z)}}{2\overline{\phi_f(z)}}, \frac{\partial G_f(z)}{\partial {z}}=\frac{{\phi'_f(z)}}{2{\phi_f(z)}},$$
%and
$$\frac{\partial}{\partial \overline{z}}\Big(\frac{\phi_f(z)}{|\phi_f(z)|}\Big)=-\frac{\phi_f(z)}{|\phi_f(z)|}\frac{\partial G_f(z)}{\partial \overline{z}},
\frac{\partial}{\partial {z}}\Big(\frac{\phi_f(z)}{|\phi_f(z)|}\Big)=\frac{\phi_f(z)}{|\phi_f(z)|}\frac{\partial G_f(z)}{\partial {z}}.$$

We have

$$\frac{\overline{\partial }\phi}{\partial \phi}=\frac{\overline{\partial }(\phi_g\circ\phi)}{\partial (\phi_g\circ\phi)}=\frac{S_\phi'(G_f(z))+iT_\phi'(G_f(z))-1}{S_\phi'(G_f(z))+iT_\phi'(G_f(z))+1}\frac{\overline{\partial }G_f}{\partial G_f}.$$ \end{proof}

%Lemma \ref{beltrami} implies the conjugacy $\phi$ satisfies the Branner-Hubbard type motion:

\subsection{Minimal dilatation and uniqueness}

%In the fundamental annulus $A=\{G_f^0<G_f(z)<G_f^N\}$, the critical levels are labeled by $G_f^0<G_f^1<\cdots<G_f^N=dG_f^0$.
Let  $\phi$ be the  conjugacy between $(f,X(f))$ and $(g, X(g))$.
 %It's equivalent class is determined by the values of $S_\phi, T_\phi$ at the points $G_f^0,G_f^1,\cdots,G_f^{N-1}$. Set $G_g^j=S_\phi(G_f^j)$ and $\Theta_\phi^j=T_\phi(G_f^j)$.
Set
$$\tau^*_j=\frac{S_\phi(\ell_j)-S_\phi(\ell_{j-1})}{\ell_j-\ell_{j-1}}+i\frac{T_\phi(\ell_j)-T_\phi(\ell_{j-1})}{\ell_j-\ell_{j-1}},\
\delta_{[\phi]}^j=\Big|\frac{\tau^*_j-1}{\tau^*_j+1}\Big|, \  j=1,\cdots,N.$$

%\begin{lem}\label{greater} Let $\phi$ be a quasiconformal conjugacy between   $(f,X(f))$ and $(g, X(g))$. Then we have
% $\|\mu_\phi\|\geq\max_{1\leq j\leq N}\delta_{[\phi]}^j.$
%\end{lem}
%\begin{proof} Let $\varphi_t:\mathbb{C}\rightarrow\mathbb{C}$ be the holomorphic family of  quasiconformal maps, solving the Beltrami equation
%$$\frac{\overline{\partial }\varphi_t}{\partial \varphi_t}=t\frac{\mu_\phi}{\|\mu_\phi\|}, |t|<1,$$
%normalized so that $f_t=\varphi_t\circ f\circ\varphi_t^{-1}\in\mathcal{P}_d$ and $\varphi_0=id$.
%Then $(g, X(g))$ is conformally conjugate to $(f_{\|\mu_\phi\|}, X(f_{\|\mu_\phi\|}))$ and
%$G_{f_{\|\mu_\phi\|}}\circ \phi_{\|\mu_\phi\|}=G_g\circ \phi$ and $T_\phi(G_f^j)-T_\phi(G_f^{j-1})=
%T_{\varphi_{\|\mu_\phi\|}}(G_f^j)-T_{\varphi_{\|\mu_\phi\|}}(G_f^{j-1})$ for $1\leq j\leq N$.

%For each $j=0,\cdots, N-1$, we choose a point $z$ such that $G_f(z)=G_f^j$. Then the function $\Phi_j:\mathbb{D}\rightarrow\mathbb{C}-\overline{\mathbb{D}}$
%defined by $\Phi_j(t)=\phi_{f_t}(z)$ is holomorphic. Let $I=[0,\|\mu_\phi\|]$ and $\gamma$ be the geodesic in $\mathbb{C}-\overline{\mathbb{D}}$ homotopic to  $\Phi_j(I)$ rel the two end points, and let $\rho$ be the hyperbolic metric in $\mathbb{C}-\overline{\mathbb{D}}$.  By Schwarz lemma, we have
%$$d_\mathbb{D}(0,\|\mu_\phi\|)\geq\int_{I}\Phi_j^*\rho\geq \int_{\gamma}\rho.$$
%Equivalently, $$\|\mu_\phi\|\geq \tanh\Big(\frac{1}{2}\int_{\gamma}\rho\Big)= \delta_{[\phi]}^j.$$
%\end{proof}

%Then we have
%$$\delta_{[\phi]}((f,X(f)), (g,X(g)))\leq \max_{1\leq j\leq N}\delta_{[\phi]}^j.$$
Note that $\tau^*_j$ depends only on the homotopy class $[\phi]$.
%Now, we formulate the following:

\begin{thm}\label{main} Let $\phi$ be a conjugacy between   $(f,X(f))$ and $(g, X(g))$. Then
$$\delta_{[\phi]}= \max_{1\leq j\leq N}\delta_{[\phi]}^j.$$
The extremal quasiconformal conjugacy is unique if and only if
$$\delta_{[\phi]}^1=\cdots=\delta_{[\phi]}^N.$$
\end{thm}

%\begin{rmk} There are two possibilities about the extremal map:

%1.   There is no homotopy  class where the extremal map is unique.

%2.  There are infinitely many homotopy  classes where the extremal maps are unique.
%\end{rmk}

To prove  Theorem \ref{main},  we need the following result:

\begin{thm}[Reich-Strebel]\label{uni} Let $\phi: R_1\rightarrow R_2$ be a quasiconformal  map  between hyperbolic Riemann surfaces. If there exist a
 holomorphic quadratic differential $q$ on $R_1$ with $\int_{R_1}|q|<+\infty$ and a number $0\leq k<1$, such that $\mu_\phi=k{\overline{q}}/{|q|}$,
%$$\frac{\overline{\partial}\phi}{\partial \phi}=k\frac{\overline{q}}{|q|},$$
then every quasiconformal  map $\psi: R_1\rightarrow R_2, \psi\neq\phi$ homotopic to $\phi$ modulo the boundary satisfies $\|\mu_\psi\|>\|\mu_\phi\|$.
\end{thm}

The proof is as in \cite{RS}, p.380.

\

\noindent{\it Proof of Theorem \ref{main}.}  Note that every map $\psi\in [\phi]$ is determined by the restrictions of  $S_\psi, T_\psi$ in $[\ell_0,\ell_N]$.

We construct a quasiconformal conjugacy $\psi\in[\phi]$ such that
$$(S_\psi(\ell_j),T_\psi(\ell_j))=(S_\phi(\ell_j),T_\phi(\ell_j)), j=0,\cdots,N,$$ and
$S_\psi,T_\psi$ are linear in each interval $(\ell_{j-1},\ell_j)$.
 Then by Lemma \ref{beltrami},
 $$\frac{\overline{\partial }\psi}{\partial \psi}=\sum_{k\in \mathbb{Z}}\sum_{j=1}^N \frac{\tau^*_j-1}{\tau^*_j+1}\chi_{f^k(A_j)}\frac{\overline{\partial }G_f}{\partial G_f},$$
where $\chi_E$ is the characteristic function, defined so that $\chi_E(x)=1$ if $x\in E$ and  $\chi_E(x)=0$ if $x\notin E$.
%Then we have $\|\mu_\psi\|=\max_{1\leq j\leq N}\delta_{[\phi]}^j$. Thus $\delta_{[\phi]}\leq \max_{1\leq j\leq N}\delta_{[\phi]}^j.$
%By Lemma \ref{greater}, we actually have $\delta_{[\phi]}= \max_{1\leq j\leq N}\delta_{[\phi]}^j$.
%On the other hand, t

The map $\psi$  satisfies: $\mu_\psi|_{A_j}=k_j\overline{q_j}/|q_j|$, where $k_j=\delta_{[\phi]}^j$ and
$$q_j=\frac{\overline{\tau^*_j}-1}{\overline{\tau^*_j}+1}\Big(\frac{\partial G_f}{\partial z}dz\Big)^2$$
 is a holomorphic and integrable quadratic differential on
$A_j$. By Theorem \ref{uni}, the map $\psi|_{A_j}$ is the unique extremal map on $A_j$ modulo the boundary $\partial A_j$.
%So any quasiconformal conjugacy  $\varphi\in[\phi]$ satisfies $\|\mu_\varphi\|=\max_{1\leq j\leq N}\|\mu_\varphi|_{A_j}\|
%\geq\max_{1\leq j\leq N}\delta_{[\phi]}^j.$
This implies $\delta_{[\phi]}= \max_{1\leq j\leq N}\delta_{[\phi]}^j$. Further,  if $\delta_{[\phi]}^1=\cdots=\delta_{[\phi]}^N$, then
$\psi$ is the unique extremal map.

If  $\|\mu_\psi|_{A_k}\|<\delta_{[\phi]}$ for some $k$, then one can deform $\psi|_{A_k}$ to another map $\varphi$ modulo the boundary with $\|\mu_\varphi|\|\leq\delta_{[\phi]}$. This yields another quasiconformal conjugacy homotopic to $\psi$ and with the same dilatation as $\psi$.
\hfill $\Box$

\subsection{Appendix: The Teichm\"uller space of $(f,X(f))$}

The general Teichm\"uller theory of rational maps is developed in \cite{MS}. It's shown that the Teichm\"uller space of $(f,X(f))$ is isomorphic to $\mathbb{H}^N$. Here we precise this statement by Theorem \ref{main}.

We begin with the definition of the Teichm\"uller space of $(f,X(f))$ following \cite{MS}. Let $QC(f,X(f))$ be the set of quasiconformal self-conjugacies  of $(f,X(f))$.
$QC_0(f,X(f))\subset QC(f,X(f))$  consists  of those conjugacies $h$ isotopic to the identity in the following sense:
 there is a family $(h_t)_{t\in[0,1]}\subset QC(f,X(f))$ with $h_0=id, h_1 = h$ such that $(t, z)\rightarrow (t, h_t(z))$ is a
homeomorphism from $[0, 1] \times \mathbb{\widehat{C}}$ onto itself.  The Teichm\"uller space ${\rm Teich}(f,X(f))$ is the set of equivalence classes $[((g, X(g)), \phi)]$ of pairs $((g, X(g)), \phi)$, where  $g\in \mathcal{P}_d$, $\phi$  is a quasiconformal conjugacy between $(f,X(f))$ and $(g,X(g))$, and we say  $((g_1, X(g_1)), \phi_1)$ and $((g_2, X(g_2)), \phi_2)$ are equivalent  if there is a conformal conjugacy $\psi$ between $(g_1, X(g_1))$ and $(g_2, X(g_2))$ such that $\phi_2^{-1}\circ \psi\circ\phi_1\in QC_0(f,X(f))$.
The Teichm\"uller metric between $\zeta_i=[((g_j, X(g_j)), \phi_j)], j=1,2$ is defined by
$$d(\zeta_1,\zeta_2)=\inf \log K(\phi),$$
where the infimum is over all quasiconformal conjugacies between $(g_1, X(g_1))$ and $(g_2, X(g_2))$,
homotopic to $\phi_2\circ \phi_1^{-1}$.  We remark that this metric (instead of its half) can make the map in Theorem \ref{teich} isometric. %Remark, here the d(\zeta_1,\zeta_2)

Give any pair of points $a=(a_1,\cdots,a_N)$, $b=(b_1,\cdots,b_N)\in \mathbb{H}^N$, we define the distance $d_{\mathbb{H}^N}(a,b)$ by
$$d_{\mathbb{H}^N}(a,b)=\max_{1\leq j\leq N}d_\mathbb{H}(a_j,b_j).$$
%where the infimum is over all representatives of the markings $[((g_j, X(g_j)), \phi_j)]$.

For $\tau=(\tau_1,\cdots,\tau_N)\in\mathbb{H}^N$,  let  $\phi_\tau:\mathbb{C}\rightarrow\mathbb{C}$ solve the Beltrami equation
$$\frac{\overline{\partial }\phi_\tau}{\partial \phi_\tau}=\sum_{k\in \mathbb{Z}}\sum_{j=1}^N \frac{\tau_j-1}{\tau_j+1}\chi_{f^k(A_j)}\frac{\overline{\partial }G_f}{\partial G_f}.$$
The holomorphic family of quasiconformal maps is normalized so that $\phi_{(1,\cdots,1)}=id$ and $f_\tau=\phi_\tau\circ f\circ\phi^{-1}_\tau \in \mathcal{P}_d$.
 By the proof of Theorem \ref{main}, we have (compare  \cite{DP1}, Lemma 5.2)

\begin{thm}\label{teich} The map
\begin{equation*}
\Phi:\begin{cases}
 \mathbb{H}^N\rightarrow {\rm Teich}(f,X(f)),\\
\tau\mapsto [((f_\tau, X(f_\tau)),\phi_\tau)].
\end{cases}
\end{equation*}
is biholomorphic and isometric.
\end{thm}

\end{document}